\documentclass[10pt]{amsart}
\usepackage{graphicx}                                    
\usepackage{amssymb}
\usepackage{MnSymbol}
\newtheorem{theorem}{Theorem}[section] 
\newtheorem{lemma}[theorem]{Lemma}
\newtheorem{conjecture*}[theorem]{Conjecture}
\newtheorem*{remark*}{Remark}
\newtheorem*{example*}{Example}
\newtheorem{proposition}[theorem]{Proposition}

\newtheorem{corollary}[theorem]{Corollary}

\newtheoremstyle{notauto}{}{}{\itshape}{}{\bfseries}{.}{0.5em}{\thmnote{#3}}
\theoremstyle{notauto}

\theoremstyle{definition}
\newtheorem{definition}[theorem]{Definition}

\begin{document}
    
    \baselineskip14pt
    
    \title[The topology on a finite set ] 
    {On The Number of Topologies On A Finite Set}
    \author{M.Yas\.{I}r KIZMAZ}
    
    \address{Department of Mathematics, Middle East Technical University, Ankara 06531, Turkey}
    \email{yasir@metu.edu.tr}
    

    \begin{abstract} We denote the number of distinct topologies which can be defined on a set $X$ with $n$ elements by
        $T(n)$. Similarly, $T_0(n)$ denotes the number of distinct $T_0$ topologies on the set $X$. In the present paper, we prove that for any prime $p$, $T(p^k)\equiv k+1 \ (mod \ p)$, and that for each natural number $n$ there exists a unique $k$ such that $T(p+n)\equiv k \ (mod \ p)$. We calculate $k$ for $n=0,1,2,3,4$. We give an alternative proof for a result of Z. I. Borevich to the effect that $T_0(p+n)\equiv T_0(n+1) \ (mod \ p)$.
    \end{abstract}
    \maketitle

    \section{\Large introduction}\label{Intro}
    
    Given a finite set $X$ with $n$ elements, let $\mathfrak{T} (X)$ and $\mathfrak{T}_0(X)$  be the family of all topologies on $X$ and the family of all $T_0$ topologies on $X$, respectively.  We denote the cardinality of $\mathfrak{T} (X)$ by $T(n)$ and the cardinality of $\mathfrak{T}_0(X)$ by $T_0(n)$. There is no simple formula giving $T(n)$ and $T_0(n)$. 
    
    Calculation of these sequences by hand becomes very hard for $n\geq 4$. The online encyclopedia of N. J. A. Sloane \cite{N.J.A} gives the values of $T(n)$ and $T_0(n)$ for $n\leq 18$. For a more detailed discussion of results in the literature, we refer to the article by Borevich \cite{B}. 
    
    In the present paper, we prove that for any prime $p$, $T(p^k)\equiv k+1 \ (mod \ p)$, and that for each natural number $n$ there exists a unique $k$ such that $T(p+n)\equiv k \ (mod \ p)$. We calculate $k$ for $n=0,1,2,3,4$. We give an alternative proof for a result of Z. I. Borevich \cite{B2} to the effect that $T_0(p+n)\equiv T_0(n+1) \ (mod \ p)$. Our results depend on basic properties of group action, which is the first time used in this problem as far as author is aware.
    
    \section{\Large Main Results}
    Let $G$ be a group acting on the finite set $X$. Then the action of $G$ on $X$ can be extended to the action of $G$ on $\mathfrak{T}(X)$ by setting $g\tau=\{gU\mid U\in \tau\}$ where $\tau\in \mathfrak{T}(X) $ and $gU=\{gu\mid u\in U\}$. Now set $Fix(\mathfrak{T}(X))=\{\tau\in \mathfrak{T}(X)\mid g\tau=\tau \ \text{for all} \ g\in G\}$. Notice that if $G$ is a $p$-group, $|\mathfrak{T}(X)|\equiv |Fix(\mathfrak{T}(X))| \ (mod \ p)$ as every non-fixed element of $\mathfrak{T}(X)$ has orbit of length a positive power of $p$. For a given topology $\tau\in \mathfrak{T}(X)$ and $x\in X$, we denote the intersection of all open sets of $\tau$ including $x$ by $O_x$. Note that $O_x\in \tau$ as we are in the finite case. 
    
    \begin{definition}
        A base $\mathfrak B$ of a topology $\tau$ is called a minimal base if any base of the topology contains $\mathfrak B$.
    \end{definition}
    \begin{example*}
    	Let $X=\{a,b,c\}$ be a set and let $\tau$ be a topology on $X$ with $\tau=\{\emptyset,\{a\},\{a,b\},\{a,c\},\{a,b,c\} \}$. Now if $\mathfrak B$ is an arbitrary base for $\tau$ then we have $U=\bigcup_{x\in U} B_x $ where $U\in \tau$ and $B_x\in \mathfrak B$ such that $x\in B_x\subseteq U$. Thus, we observe that $\mathfrak B'=\{\{a\},\{a,b\},\{a,c\}\}\subseteq \mathfrak B$. Moreover, $\mathfrak B'$ is also a base for $\tau$ which implies that $\mathfrak B'$ is a minimal base for $\tau$.
    \end{example*}
    \begin{proposition}
        Let $\tau\in \mathfrak T(X)$ and $M_\tau$ be the set of all distinct $O_x$ for $x\in X$.
        A base $\mathfrak B$ of $\tau$ is a minimal base if and only if $\mathfrak B=M_\tau$.
    \end{proposition}
    By the proposition, we extend the definition: a base $\mathfrak B$ on $X$ is minimal if $\mathfrak B=M_\tau$ where $\tau$ is the topology generated by $\mathfrak B$.

    \begin{lemma}\label{lemma}
        $\tau\in Fix(\mathfrak {T}(X))$ if and only if $M_\tau$ is $G$-invariant.
    \end{lemma}
    
    \begin{proof}
        Let $\tau\in  Fix(\mathfrak {T}(X)) $ and $O_x\in M_\tau$. We need to show that $gO_x\in M_\tau$. Let $gx=y$ for $g\in G$. As $g\tau=\tau$, $gO_x\in \tau$. Thus, $gO_x$ is an open set containing the element $gx=y$. Hence, $O_{y}\subseteq gO_x$. Since $g^{-1}y=x$, we can show that $O_x\subseteq g^{-1}O_y$ which forces $gO_x=O_y$. So, $  gO_x\in  M_\tau \ \forall g\in G$. Now assume that $M_\tau$ is $G$-invariant. For each  $U\in\tau$, we have $U=\bigcup_{x\in U} O_x $. Then $gU\in \tau$  as $gO_x\in M_{\tau}$ for all $g\in G$, and hence $\tau\in Fix(\mathfrak {T}(X))$.
    \end{proof}
    
    \begin{theorem}\label{$p^k$ theorem}
        Let $p$ be a prime and let $k$ be a natural number. Then $T(p^k)\equiv k+1 \ (mod \ p).$
    \end{theorem}
    \begin{proof}
        Without loss of generality, let $X$ be the cyclic group of order $p^k$, that is, $X={C}_{p^k}$. Clearly $X$ acts on $X$ by left multiplication. By extending this action, $X$ acts on $\mathfrak{T}(X)$. Notice that $|\mathfrak{T}(X)|\equiv |Fix(\mathfrak{T}(X))| \ (mod \ p)$ as $X$ is a $p$-group.
        It is left to show that $Fix(\mathfrak{T}(X))$ has $k+1$ elements. 
        
        Let $\tau\in Fix(\mathfrak{T}(X))$ and let $O_x,O_y\in M_\tau$ for $x,y\in X$. Then  $(yx^{-1})O_x$ is an open set including $y$. Hence, $O_y\subseteq (yx^{-1})O_x$ which means $|O_y|\leq |O_x|$. The other inclusion can be done similarly so $|O_x|=|O_y|$ for all $x,y\in X$. Now if $m\in O_x\cap O_y$, then $O_m\subseteq O_x\cap O_y$. As their orders are equal, we must have $O_x=O_y$ or $O_x\cap O_y=\emptyset$. Thus, $X$ is a disjoint union of elements of $M_\tau$, that is, $X=\bigcupdot_{O_x\in M_{\tau}} O_x $. It follows that $|X|=|M_{\tau}||O_x|$ for any $x\in X$.
        
         Note that $X$ also acts on $M_\tau$ by Lemma 2.3. Let $e$ be the identity element of $X$ and $Stab(O_e)$ be the stabilizer of $O_e$ in the action of $X$ on $M_{\tau}$. Then we obtain that $|X:Stab(O_e)|=|M_{\tau}|$ as the action is transitive. It follows that $|Stab(O_e)|=|O_e|$ as $|X|=|M_{\tau}||O_e|$ also holds. Since our action is induced from left multiplication in the group $X$, the product $Stab(O_e)O_e$ is exactly set multiplication in the group $X$. Then we obtain that $O_e=Stab(O_e)O_e\supseteq Stab(O_e)e=Stab(O_e)$. But the equality $|O_e|=|Stab(O_e)|$ forces that $O_e=Stab(O_e)$.
         
          Hence, the set $M_\tau$ is the set of left cosets of a subgroup of $X$. Since the chosen topology $\tau$ from $Fix(\mathfrak{T}(X))$ uniquely determines $M_\tau$ and $M_\tau$ uniquely determines a subgroup $O_e$ of $X$, we have an injection from $Fix(\mathfrak{T}(X))$ to the set of all subgroups of $X$. Conversely, for a subgroup $H$ of $X$, the set of left cosets of $H$ forms an $X$-invariant minimal base for a topology. The topology $\tau$ generated by this base is an element of $Fix(\mathfrak{T}(X))$ by Lemma 2.3. Hence, the cardinality of $Fix(\mathfrak{T}(X))$ is equal to the number of the subgroups of $X$, which is $k+1$.
    \end{proof}  
    By applying the same method in the above proof, we can also show that $T_0(p^k)\equiv 1 \ (mod \ p)$. Actually, Z. I. Borevich proved a more general result  about $T_0(n)$. Now we establish an alternative proof for the theorem of Z. I. Borevich.
    
    \begin{theorem}[Z. I. Borevich] \label{Borevich thm}
        Let $p$ be a prime. If $k \equiv l \ (mod \ p - 1)$, then $T_0(k)\equiv T_0(l) \ (mod \ p)$.
        \end {theorem}
        \begin{proof}
            It is equivalent to show that $T_0(n+p)\equiv T_0(n+1) \ (mod \ p)$ for an integer $n\geq 0$.
            Let $C=C_p$ be the cyclic group of order $p$ and $N$ be a set with $n$ elements. Without loss of generality, let $X$ be the disjoint union of $C$ and $N$ so that the cardinality of $X$ is $p+n$. We define the action of $C$ on $X$ in the following way: for $c\in C$ and for $x\in X$, $c*x=cx$ if $x\in C$ and $c*x=x$ if $x\in N$. Then the action of $C$ on $X$ can be extended to the action of $C$ on $\mathfrak{T_0}(X)$. As $|\mathfrak{T_0}(X)|\equiv |Fix(\mathfrak{T_0}(X))| \ (mod \ p)$, it is left to show that $|Fix(\mathfrak{T_0}(X))|=T_0(n+1)$.
            
             Let $\tau\in Fix(\mathfrak{T_0}(X))$. Pick $x,y\in C$ and $a\in N$ where $x\neq y$. We know that $yx^{-1}O_x=O_y$. Then we can observe that $O_x\cap N=O_y\cap N$ as $N$ is fixed by $C$. Similarly, $O_x\cap C$ and $O_y\cap C$ are disjoint or equal. But we can not have $O_x=O_y$ as it is a $T_0$ topology. Hence, $(O_x\cap C) \cap (O_y\cap C)=\emptyset$. Then we obtain that $O_x\cap C=\{x\}$, and hence $O_x=\{x\}\cup S$ for a subset $S$ of $N$. On the other hand, $x(C\cap O_a)=C\cap O_a$, which forces that $C\cap O_a$ is either $C$ or $\emptyset$, and hence $O_a$ is either $C\cup T$ or $T$ for a subset $T$ of $N$.
             Set $Y=\{x\}\cup N$  then for each $\tau\in Fix(\mathfrak{T_0}(X))$, we have subspace topology $\tau_Y=\{U\cap Y \mid U\in \tau \}$ on $Y$. Note that $\tau_Y$ is a $T_0$ topology as it is the subspace topology on $Y$ induced from $\tau$.
              
              Now our aim is to show that the mapping  $\tau\mapsto \tau_Y$ is a bijection from $Fix(\mathfrak{T_0}(X))$ to ${T_0}(Y)$. Let $\pi,\tau\in Fix(\mathfrak{T_0}(X)) $ such that $\tau_Y=\pi_Y$. Then we have $O_t\cap Y=O'_t\cap Y$ for all $t\in Y$ where $O_t\in M_\tau$ and $O'_t\in M_\pi$. Note that $O_x\subseteq Y$ by the properties deduced in previous paragraph, and hence we obtain $O_x=O'_x$. We also have the equality $O_a\cap Y=O'_a\cap Y$ for $a\in N$, which forces that $O_a\cap C=O'_a\cap C$ is either $C$ or $\emptyset$ by previous paragraph. Since $O_a\cap N=O'_a\cap N$,  we get $O_a=O'_a$ for $a\in N$. Then it follows that $O_t=O'_t$ for all $t\in Y$. Since $O_z=zx^{-1}O_x$ for $z\in C$, we obtain that $O_t=O'_t$ for all $t\in X$, which implies the equality $\tau=\pi$. Thus, our map is one to one.

               Now, let $\pi\in \mathfrak{T}_0(Y)$ and let $M_{ \pi}$ be the minimal base of $\pi$. Set $$O_t=\begin{cases}
                  tx^{-1}O'_{x} & \ if \ t\in C \\ 
                O'_{t} & \ if \ t\in N \ and \ x\notin O'_{t}\\
                C\cup O'_{t} & \ if \ t\in N \ and \ x\in O'_{t}
                \end{cases} $$ for each  $O'_{t}\in M_{\pi}$. We need to show that  $\mathfrak{B}=\{O_t\mid t\in X\}$ is a minimal base and the topology $\tau$ generated by $\mathfrak B$ is a $T_0$ topology. To show that it is a minimal base, we need to show that if $a\in O_t$ then $O_a\subseteq O_t$ for any $a,t\in X$. In fact, we will observe that $O_a\subset O_t$ which shows that $\tau$ is a $T_0$ topology. It can be done case by case but here we present only nontrivial cases. Fix $a\in N$ and let $a\in O_x$. Note that $a\in O'_x$ as $O'_x=O_x$, and hence we obtain $O'_a\subseteq O'_x$. If $x\in O'_a$ then we have $O'_a=O'_x$ which is not possible as $\pi$ is a $T_0$ topology. Thus, $x\notin O'_a$ which implies $O_a=O'_a$ by our setting. Then we obtain that $O_a\subset O_x$. Now assume that $x\in O_a$. Then we get that $x\in O'_a$ otherwise $x\notin O_a$ by our setting. Thus, we obtain $O'_x\subset O'_a$ in a similar way. It follows that $O_x\subset O_a$ as $O_x=O'_x$. 
                As a result  $\tau \in Fix(\mathfrak{T_0}(X))$ as $\mathfrak{B}$ is a $C$-invariant minimal base. Moreover, the equality $ \tau_Y =\pi$ holds, which concludes that the mapping $\tau\mapsto \tau_Y$  from $Fix(\mathfrak{T_0}(X))$ to $\mathfrak{T_0}(Y)$ is a bijection, which completes the proof. 
        \end{proof}

        \begin{corollary}
            $T_0(p^k)\equiv 1 \ (mod \ p)$ where $k$ is a natural number and $p$ is a prime number.
        \end{corollary}
        
        The proof of next theorem is similar to the proof of Theorem \ref{Borevich thm} in the sense that both use the same technique. For clarity, we repeat some arguments.
        \begin{theorem}
            For each natural number $n$, there exists a unique integer $k$ such
            that $T(p+n)\equiv k \ (mod \ p)$ for all primes $p$.
        \end{theorem}
        \begin{proof}
            If $n=0$, $T(p)\equiv 2 \ (mod \ p)$  by Theorem $2.4$. Now we can assume that $n>0$.
            Let $C=C_p$ be the cyclic group of order $p$ and $N$ be a set with $n$ elements. We define the action of $C$ on $X$ as in the proof the previous theorem.  Then the action of $C$ on $X$ can be extended to the action of $C$ on $\mathfrak{T}(X)$. As $|\mathfrak{T}(X)|\equiv |Fix(\mathfrak{T}(X))| \ (mod \ p)$, it is left to show that $|Fix(\mathfrak{T}(X))|$ does not depend on the choice of prime $p$.
            
             Due to Lemma 2.3, $|Fix(\mathfrak{T}(X))|$ is equal to the number of $C$-invariant minimal bases on $X$. Let $x,y\in C$. Notice that $O_x$ completely determines $O_y$ as $(yx^{-1})O_x=O_y$. Then $|O_x\cap C|=|O_y\cap C|$ as $yx^{-1}(O_x\cap C)=O_y\cap C$. Now it is easy to see that $O_x\cap C$ and $O_y\cap C $ are either disjoint or equal. As cardinality of $C$ is prime, $O_x\cap C$ is either $\{x\}$ or $C$. If  $a\in N$ then we must have $gO_a=O_a$ for all $g\in C$ as $ga=a$. Hence, $O_a\cap C$ is either $\emptyset$ or $C$. Thus, the elements of $C$ have no contribution to the number of possible minimal bases. Now we know that such $k$ exists. If $k'$ is also such an integer then $k\equiv k' \ (mod \ p)$ for all primes $p$. Hence $k-k'$ is divisible by all primes which forces $k-k'=0$. 
        \end{proof}    
        By the previous theorem, we see that $k$ is uniquely determined by $n$, which leads the following definition.
        \begin{definition}
            We define the integer sequence $k(n)$ for $n\geq 0$ to be the unique number such that $k(n)\equiv T(p+n) \ (mod \ p)$ for all primes $p$. 
        \end{definition}
        The sequence $k(n)$ has appeared in the online encyclopedia of N. J. A. Sloane [2] with reference number (A265042) after this article appeared in arxiv.
        The proof of the theorem also gives an algorithm to calculate $k(n)$ for a given $n$.
        
        \begin{corollary}
            $T(p+1)\equiv 7 \ (mod \ p)$ for all primes $p$, that is, $k(1)=7$.
            \end {corollary}
            
            \begin{proof}
                By following the proof the Theorem 2.7,  it can be easily counted that there are exactly seven $C$-invariant minimal bases. Then the result follows.
            \end{proof}
            However, it is difficult to count $C$-invariant minimal bases for larger $n$ to determine $k(n)$. We develop a new method to calculate $k(n)$ for larger $n$. But the method requires knowing values of $T(s)$ for some $s$.
            
            \begin {theorem} \label{inequality}
            The sequence $k(n)$ satisfies the following inequality $T(n+1)+T_0(n+1)\leq k(n) < 2 T(n+1)$ for $n\geq 1$.    
            \end {theorem}
            
            \begin{proof}
                We follow the proof of Theorem 2.7. We have $k(n)=|Fix(\mathfrak{T}(X))|$ where $|X|=n+p$ and $ |C|=p$. Thus, we need to count $C$-invariant minimal bases. According to the proof, we have two main cases.
                
                \textit{Case1:}  $O_x\cap C=C$ for $x\in C$.\\ Then $O_y\cap C=C$ for all $y\in C$ and $O_x=O_y$ for all $x,y\in C$. If $a\in N$ then $O_a\cap C$ is either $C$ or $\emptyset$. Hence we can see the whole $C$ as one element, that is, pass to the quotient $\overline{X}=X/\sim$ where $x\sim y$ if $x,y\in C$ then there is a one to one correspondence between $C$-invariant topologies and quotient   topologies. It follows that we have exactly $T(n+1)$ possible sub-cases.
                
                \textit{Case2:} $O_x\cap C=\{x\}$ for $x\in C$.\\ 
                Let $x,y\in C$ then when we define $O_x$, $O_y$ is completely determined as $O_y=yx^{-1}O_x$. Moreover, we have $O_x\cap N=O_y\cap N$ as $C$ acts trivially on $N$. We also have $O_a\cap C$ is either $C$ or $\emptyset$. Now set $Y=N\cup \{x\}$ then the mapping $\tau \to \tau_Y$ is one to one from $Fix(\mathfrak{T}(X))$ to $\mathfrak{T}(Y)$.  Hence we can have at most $T(n+1)$ possible sub-cases. But we can not have exactly $T(n+1)$ possible sub-cases as the given map is not onto. To see this, set $O'_x=\{x,a\}$ for $a\in N$ and  $O'_a=\{x,a\}$ for a topology $\pi \in \mathfrak{T}(Y)$. Let $\tau\in Fix(\mathfrak{T}(X))$ such that $\tau_Y=\pi$. Note that $O'_a= O_a\cap Y$ and $O_a$ must be fixed by the action of $C$. Thus, we have $C\cup \{a\}\subseteq O_a$. Since $O_x=O'_x=\{x,a\}$, we obtain that $O_a\subseteq \{x,a\}$. Then we obtain that $O_a=\{a\}$ as it must be fixed by $C$, which is a contradiction. Thus, the map is not onto. Then we get that $k(n)<2T(n+1)$. Moreover, the mapping is onto from $Fix(\mathfrak{T}_0(X))\subseteq Fix(\mathfrak{T}(X)) $ to $\mathfrak{T}_0(Y)\subseteq \mathfrak{T}(Y)$.    (See the proof of Theorem \ref{Borevich thm} for details of how the map $\tau \to \tau_Y$ is one to one. It also shows why this map is onto when the target set is $\mathfrak{T}_0(Y)$.) Thus, we have at least $T_0(n+1)$ sub-cases, and hence $T(n+1)+T_0(n+1)\leq k(n) < 2 T(n+1).$
            \end{proof}
            \begin{corollary}
            	$\lim\limits_{n\to \infty}^{}\dfrac{k(n)}{T(n+1)}=2$.
            \end{corollary}
            \begin{proof}
            	We have $1+\frac{T_0(n+1)}{T(n+1)}\leq\frac{k(n)}{T(n+1)}<2$ by Theorem \ref{inequality}. In \cite{E}, it is proved that $lim_{n\to \infty}\frac{T_0(n)}{T(n)}=1$, and hence the result follows.
            \end{proof}
            
            \begin{theorem}
                The sequence $k(n)=2,7,51,634,12623$ for $n=0,1,$ $2,3,4$ respectively.\end {theorem}
                \begin{proof}

                If $n=0$, the result follows from Theorem \ref{$p^k$ theorem}. Now assume that $n\geq 1$. We only show the calculation of $k(2)$ and rest of them follow in a similar way.
                By previous theorem, we have $T(3)+T_0(3)\leq k(2)<2T(3)$ so $48\leq k(2)<58$.
                Clearly, we have 
                $$\begin{array}{rcl}
                T(5)& \equiv & k(2) \ (mod \ 3) \\
                T(7)& \equiv & k(2) \ (mod \ 5)
                
                \end{array}$$

                 It follows that $k(2)\equiv 6 \ (mod \ 15)$ by solving the above congruence relation. We obtain that $k(2)=51$ as $48\leq k(2)<58$. For $n=3,4$, we have the same procedure.
                \end{proof} 
            
            We should note that for $n\geq5$, there is no unique solution satisfying the inequality. For example, $k(5)\in \{357593,387623,417653\}.$
                The closed form of $k(n)$ seems to be another open problem. Hence calculation of $k(n)$ for specific $n$ or some better lower and upper bounds can be seen as new problems arising from this article. 

                {}
                \vskip 0.3 cm
            \end{document}